\documentclass[12pt]{article}
\usepackage{graphicx,psfrag,epsf}
\usepackage{float, subfigure}
\usepackage{enumerate}
\usepackage{natbib}
\usepackage{amssymb, amsmath, amsfonts, amsthm}
\usepackage{mathrsfs}

\newcommand{\blind}{0}

\newtheorem{theorem}{Theorem}[section]

\newtheorem{lemma}{Lemma}[section]

\newtheorem{proposition}{Proposition}[section]

\newtheorem{corollary}{Corollary}[section]

\theoremstyle{remark}

\newtheorem{example}{\bf Example}[section]

\renewenvironment{proof}{\noindent{\it Proof.}}{\qed}

\def\d{\mbox{d}}

\def\half{\hbox{$1\over2$}}

\DeclareMathOperator*{\argmin}{argmin}

\newcommand{\RR}{\mathbb{R}}
\newcommand{\XX}{\mathbb{X}}

\begin{document}

\def\spacingset#1{\renewcommand{\baselinestretch}%
{#1}\small\normalsize} \spacingset{1}

%%%%%%%%%%%%%%%%%%%%%%%%%%%%%%%%%%%%%%%%%%%%%%%%%%%%%%%%%%%%%%%%%%%%%%%%%%%%%%

\if0\blind
{
  \title{\bf Applications of an algorithm for solving Fredholm equations of the first kind\thanks{
    Martin and Walker acknowledge NSF support with grants DMS-1611791 and DMS-1612891, respectively.}}
  \author{Minwoo Chae\\
    Department of Mathematics, Applied Mathematics and Statistics \\ Case Western Reserve University \vspace{.2cm} \\
    Ryan Martin \\
    Department of Statistics, North Carolina State University
    \vspace{.2cm} \\
    and
    \vspace{.2cm} \\
    Stephen G. Walker \\
    Department of Mathematics, University of Texas at Austin
    }
  \maketitle
} \fi

\if1\blind
{
  \bigskip
  \bigskip
  \bigskip
  \begin{center}
    {\LARGE\bf Title}
\end{center}
  \medskip
} \fi

\bigskip
\begin{abstract}
In this paper we use an iterative algorithm for solving Fredholm equations of the first kind. The basic algorithm is known and is based on an EM algorithm when involved functions are  non--negative and integrable. With this algorithm we demonstrate two examples involving the estimation of a mixing density and a first passage time density function involving Brownian motion.  We also develop the basic algorithm to include functions which are not necessarily non--negative and again present illustrations under this scenario. A self contained proof of convergence of all the algorithms employed is presented.
\end{abstract}

\noindent%

\spacingset{1.45}
\section{Introduction}  

An important problem in statistics and applied mathematics is the solution to a so-called {\em Fredholm equation of the first kind}, i.e., given a probability density function $f(x)$ on $\XX \subset \RR$ and a non-negative kernel $k(x,\theta)$ which is a probability density function in $x$ for each $\theta$, find the probability density function $p(\theta)$ such that 
\begin{equation}
\label{fred}
f(x)=\int_\Theta k(x,\theta)\,p(\theta)\,\d\theta.
\end{equation}
Such equations have a wide range of applications across a variety of fields, including signal processing, physics, and statistics.  See \cite{Ramm} for a general introduction, and \cite{Vangel} for statistical applications.  We will also consider problems where the density and non-negativity constraints are relaxed.

Given the importance of this problem, it should be no surprise that there is a substantial literature on the theoretical and computational aspects of solving \eqref{fred}; see, for example, \cite{Corduneau, Groetsch, Hansen, Morozov, Wing}.  Existence of a unique solution to \eqref{fred} is a relevant question, though we will not speak directly on this issue here.  One important case for which there is a well-developed existence theory is when the operator $p \mapsto \int k(\cdot, \theta) \, p(\theta) \,\d\theta$ is compact, a consequence of square-integrability of $k(x,\theta)$ with respect to $\d x \times \d\theta$.  In this case, there exists a singular value system $(\sigma_j, u_j, v_j)$ for which $(u_j(x))$ and $(v_j(\theta))$ are orthogonal bases in $L_2(\XX)$ and $L_2(\Theta)$, respectively, and
$$\int_\Theta k(x,\theta)\,v_j(\theta)\,\d\theta=\sigma_j\,u_j(x).$$
Then $f$ can be expressed as $f(x) = \sum_j \lambda_j u_j(x)$ for some $(\lambda_j)$, and it follows that 
\[ p(\theta) = \sum_j (\lambda_j / \sigma_j) \, v_j(\theta), \]
which exists and is square-integrable if $\sum_j (\lambda_j / \sigma_j)^2 < \infty$. Of course, the above formula for $p$ is not of direct practical value since computing all the eigenvalues and eigenfunctions is not feasible.  Computationally efficient approximations are required.  

The classic iterative algorithm for finding $p$ is given by
\begin{equation}\label{iter}
p_m(\theta)=p_{m-1}(\theta)+\int_\XX k(x,\theta)\,\big(f(x)-f_{m-1}(x)\big)\,\d x, 
\end{equation}
where
$$f_m(x)=\int_\Theta k(x,\theta)\,p_{m}(\theta)\,\d\theta.$$
Convergence properties of (\ref{iter}) are detailed in \cite{Landweber}.  An issue here is that, under the constraint that $p$ and $k(\cdot,\theta)$ are density functions, there is no guarantee that the sequence of approximations $(p_m)$ coming from \eqref{iter} are density functions. 

An alternative is to apply a discretization--based method.  That is, specify grid points $(\theta_j,x_i)$ and approximate the original problem \eqref{fred} via a discrete system
$$f(x_j)=\sum_{i} w_i\,p(\theta_i)\,k(x_j,\theta_i),$$
where the $(w_i)$ are weighting coefficients for a quadrature formula. One then solves the linear system of equations to get an approximation for $p$. See for example, \cite{Phillips}.  Of course, there is no guarantee here either that the solution $p$ will be a density.  

%A good recent review of the state of the art to solving Fredholm equations of the first kind is given in \cite{Groetsch}. 

As an alternative to the additive updates in \eqref{iter}, in this paper we focus on a multiplicative version, described in Section~2, that can guarantee the sequence of approximations are density functions.  Moreover, the theoretical convergence analysis of this multiplicative algorithm turns out to be rather straightforward for the case where there exists a unique solution to equation \eqref{fred}.  For the general case, where a solution may not exist, the asymptotic behavior of the algorithm can still be determined, and, in Section~3, we provide a characterization of this algorithm as an {\em alternating minimization} algorithm, and employ the powerful tools developed in \cite{Csiszar} to study its convergence.  The remainder of the paper focuses on three applications of this algorithm.  The first, in Section~4, is estimating a smooth mixing density based on samples $X_1,\ldots,X_n$ from the density $f$ in \eqref{fred}.  The second, in Section~5, is computing the density of the first passage time for Brownian motion hitting an upper boundary, where the boundary need not be concave.  The third, in Section~6, is solving general Fredholm equations where the density function constraints are relaxed.  Finally, some concluding remarks are given in Section~7.

\section{The algorithm}

Following \cite{Vardi}, an alternative to the algorithms described above for iteratively solving \eqref{fred} is to fix an initial guess $p_0$ and then repeat
\begin{equation}
\label{iter1}
p_m(\theta)=p_{m-1}(\theta)\, \int_\XX \frac{k(x,\theta)\,f(x)}{f_{m-1}(x)}\,\d x, \quad m \geq 1. 
\end{equation}
This algorithm was also studied in \cite{Shyamalkumar}.  First note that, if $f_0$ is well-defined, then it follows immediately from the multiplicative structure of the algorithm, and Fubini's theorem, that $p_m$ is also a density for $m \geq 1$.  This overcomes the difficulty faced with using \eqref{iter}.  Second, for a comparison with \eqref{iter}, we see that \eqref{iter1} operates in roughly the same way, but on the log--scale:
\begin{equation}
\label{iter1.log}
\log p_m(\theta)=\log p_{m-1}(\theta)+\log \int_\XX \frac{k(x,\theta)\,f(x)}{f_{m-1}(x)}\,\d x.
\end{equation}
This logarithmic version will be helpful when we discuss convergence in Section~3.

The sequence \eqref{iter1} also has a Bayesian interpretation.
%; and see \cite{Bayes}, for example,  for details on Bayesian theory. 
Indeed, if, at iteration $m$, we had seen an observation $X$ from the distribution with density $f$, then the Bayes update of ``prior'' $p_{m-1}(\cdot)$ to ``posterior'' $p_{m}(\cdot \mid X)$ would be given by
$$p_m(\theta \mid X)=p_{m-1}(\theta)\,\frac{k(X,\theta)}{f_{m-1}(X)}.$$
But without such an $X$, yet knowing $X$ comes from $f$, the natural choice now is to use the ``average update''
$$p_m(\theta)=\int_\XX p_m(\theta \mid x)\,f(x)\,\d x,$$
which is exactly \eqref{iter1}.  

For a quick proof--of--concept, consider a Pareto density $f(x) = a(x+1)^{-(a+1)}$, supported on $(0,\infty)$, with $a > 0$.  It is easy to check that $f$ is a gamma mixture of exponentials, i.e., $f(x) = \int k(x, \theta) p(\theta) \, \d \theta$, where $k(x,\theta) = \theta \,e^{-\theta x}$ is an exponential density and $p(\theta) = \Gamma(a)^{-1} \theta^{a-1} e^{-\theta}$ is a gamma density.  We can apply algorithm \eqref{iter1} to solve the above equation for $p$. Figure~\ref{fig:pareto} shows a plot of the estimate from \eqref{iter1}, with $a=5$, based on 200 iterations and $p_0$ a half-Cauchy starting density, along with the true gamma density $p(x)$.  Clearly, the approximation is quite accurate over most of the support.

\begin{figure}
\begin{center}
\scalebox{0.6}{\includegraphics{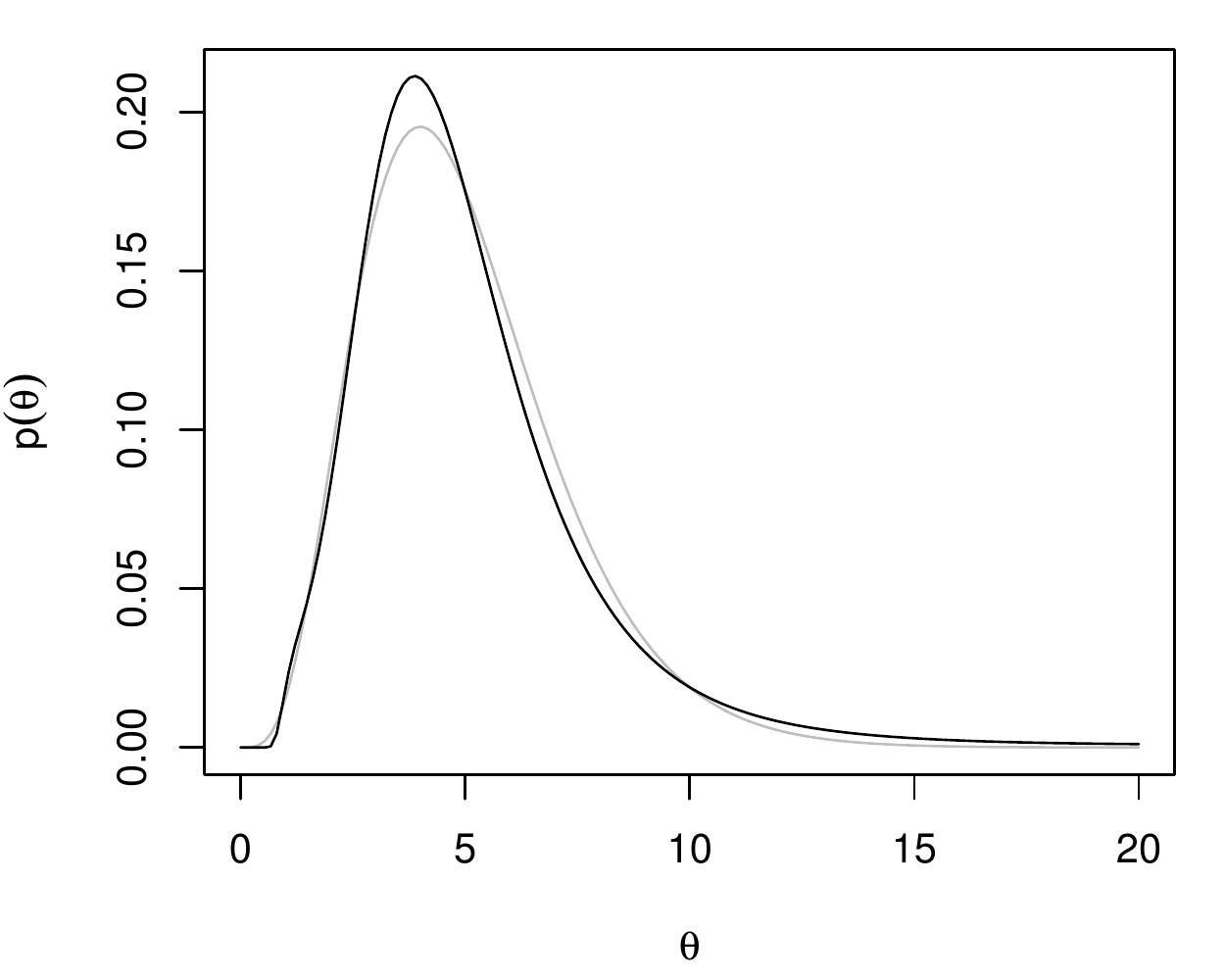}}
\caption{Plot of the mixing density $p_{200}$ (black) from \eqref{iter1} and the true gamma density (gray).}
\label{fig:pareto}
\end{center}
\end{figure}

Before formally addressing the convergence properties of \eqref{iter1}, it will help to provide some intuition as to why it should work.  The argument presented in \cite{Vardi} proceeds by considering i.i.d. samples $X_1,\ldots,X_n$ from the distribution with density $f$.  Replacing the $f(x) \,\d x$ in \eqref{iter1} with $\d\widehat F(x)$, where $\widehat F$ is the empirical distribution based on $X_1,\ldots,X_n$, gives the algorithm
\begin{equation}
\label{iter2}
\widehat{p}_m(\theta)=\widehat{p}_{m-1}(\theta)\,\int_\XX \frac{k(x,\theta)}{f_{m-1}(x)}\,\d \widehat{F}(x).  
\end{equation}
It turns out that this is precisely an EM algorithm to compute the nonparametric maximum likelihood estimator of $p$, see e.g., \cite{Laird}.  This connection to likelihood--based estimation gives the algorithm \eqref{iter2} some justification for a fixed sample $X_1,\ldots,X_n$.  The dependence on a particular sample can be removed by allowing $n \to \infty$ and applying the law of large numbers to get convergence of \eqref{iter2} to \eqref{iter1}, hence motivation for the latter.  However, despite the identification of \eqref{iter2} as an EM algorithm, no formal convergence proof has been given that the iterates $\widehat p_m$ in \eqref{iter2} converge to the nonparametric maximum likelihood estimator, but see \cite{Chae}.  

The above argument giving intuitive support for algorithm \eqref{iter1} is not fully satisfactory because it does not give any indication that the algorithm will converge to a solution of \eqref{fred}.  For a more satisfactory argument, note that, {\em if the algorithm converges} to a limit $p_\infty$, then we must have 
\[ \int_\XX \frac{k(x, \theta)}{f_\infty(x)} \, f(x) \,\d x = 1 \quad \forall \; \theta \in \text{supp}(p_\infty), \]
where $f_\infty(x) = \int k(x,\theta) \, p_\infty(\theta) \,\d \theta$.  According to Lemma~2.1 in \cite{Patilea} or Lemma~2.3 in \cite{Kleijn}, the above condition implies that 
\begin{equation}
\label{projection}
D(f,f_\infty) = \inf_P D(f,f_P), 
\end{equation}
where $D(f,g) = \int \log(f/g) \, f \,\d x$ is the Kullback--Leibler divergence and the infimum is over all densities of the form $\int k(x,\theta) \, \d P(\theta)$ for some probability measure $P$ on $\Theta$.  So, if there exists a solution to \eqref{fred}, then the limit $p_\infty$ would have to be one of them.  Even if there is no solution to \eqref{fred}, $p_\infty$ will be such that the corresponding $f_\infty$ is closest to $f$ in the Kullback--Leibler sense.  To make this argument fully rigorous, we need to establish that algorithm \eqref{iter1} does indeed converge.

%The convergence of (\ref{iter1}) is argued by \cite{Vardi} as being a law of large numbers limit of an EM algorithm. For suppose $X_1,\ldots,X_n$ are an i.i.d. sample from $f$. Then an EM algorithm for estimating $p$, in fact the NPMLE, is given by
%\begin{equation}%\label{iter2}
%\widehat{p}_m(\theta)=\widehat{p}_{m-1}(\theta)\,\int_\XX \frac{k(x,\theta)}{f_{m-1}(x)}\,\d \widehat{F}(x),
%\end{equation}
%where $\widehat{F}$ is the empirical distribution function of the $(X_i)$. See also equation (2.12) in \cite{Vardi}. Despite this being an EM algorithm, proof of convergence is not a given. A convincing proof is in Chae et al. (2017; in preparation).  A strong law of large numbers would now indicate that (\ref{iter2}) $\to$ (\ref{iter1}) as $n\to\infty$, hence providing motivation for (\ref{iter1}).

\section{Convergence properties}

In Section~3.4 of \cite{Vardi}, the authors consider \eqref{fred}, but with some restrictions.  The first is that $p$ is a finite measure, i.e., point masses on a finite number of atoms, and the second is that the solution $p$ is piecewise constant.  Our arguments here do not require such restrictions. 

To assess the properties of \eqref{iter}, let $p$ be a solution to \eqref{fred}.  Multiply \eqref{iter1.log} by $p(\theta)$ throughout, and then integrate over $\theta$ to get 
%\begin{equation}\label{logequal}
\[ \int p(\theta)\,\log p_m(\theta)\,\d\theta=\int p(\theta)\,\log p_{m-1}(\theta)\,\d\theta+\int p(\theta) \log \int \frac{k(x,\theta)\,f(x)}{f_{m-1}(x)}\,\d x\,\d\theta. \]
%\end{equation}
By Jensen's inequality, the last term is lower bounded by $D(f, f_{m-1})$, the Kullback--Leibler divergence of $f_{m-1}$ from $f$.  Since $D$ is non--negative, we can deduce that 
$$D(p,p_{m})\leq D(p,p_{m-1})-D(f,f_{m-1}).$$
This implies that $D(p, p_m)$ is a non--negative and non--increasing sequence, hence has a limit, say, $c \geq 0$, which, in turn, implies $f_m \to f$ strongly in the $D$ or $L_1$ sense.  Therefore, in agreement with \cite{Landweber} for the additive algorithm \eqref{iter}, we have that $p_m$ is converging to a set $\mathbb{S}$ where $D(p,\widetilde{p})=c$ for all $\widetilde{p}\in \mathbb{S}$.  Of course, if $p$ is a unique solution to \eqref{fred} then $p_m \to p$ strongly.  A similar argument for convergence of \eqref{iter1} is presented in \cite{Shyamalkumar}, with some generalizations.  

On the other hand, if \eqref{fred} does not have a solution, then, as discussed above, we expect that $p_m$ in algorithm \eqref{iter1} will converge to a density with limit $p_\infty$ such that the corresponding mixture $f_\infty$ satisfies \eqref{projection}.  This can be proved by considering  \eqref{iter1} as an \textit{alternating minimization} procedure; see, for example, \cite{Csiszar75, Csiszar, Dykstra}.

In the following, define the joint densities on $\mathbb{X}\times\Theta$,
\begin{equation}
\label{pi_q_def}
q_m(x,\theta) = p_m(\theta) k(x,\theta)\quad \text{and} \quad \pi_m(x,\theta) = \frac{f(x) q_{m-1}(x,\theta)}{\int q_{m-1}(x,\theta') \d\theta'}.
\end{equation}

%\medskip

\begin{theorem} 
\label{thm:convergence}
For an initial solution $p_0 > 0$, let $(p_m)$ be obtained via \eqref{iter1}.
Assume there exists a sequence $(p^*_s)_{s\geq 1}$ of densities such that 
$$D(f, f^*_s) \rightarrow \inf_P D(f, f_P)$$ and $D(\pi^*_s, q_m) < \infty$ for some $m \geq 0$, where $f^*_s(x) = \int k(x,\theta) p^*_s(\theta) \d\theta$ and 
\[ \pi^*_s(x,\theta) = \frac{f(x) p^*_s(\theta) k(x,\theta)}{\int p^*_s(\theta') k(x,\theta') \d\theta'}. \]
Then, $D(f, f_m)$ decreases to $\inf_P D(f,f_P)$.
\end{theorem}

%\medskip

\begin{proof}
Let $\mathcal{P}$ be the set of all bivariate densities $\pi$ on $\XX \times \Theta$ with $x$-marginal $f$, i.e., such that $\int \pi(x,\theta)\,\d\theta = f(x)$.  Similarly, let $\mathcal{Q}$ be the set of all bivariate densities $q$ on $\XX\times\Theta$ such that $q(x,\theta) = k(x,\theta)\, p(\theta)$ for some density $p$.
Note that $q_m$ and $\pi_m$ in \eqref{pi_q_def} satisfy $q_m \in \mathcal{Q}$ and $\pi_m \in \mathcal{P}$.

We first claim that $(q_m)$ and $(\pi_m)$ are obtained by the alternating minimization procedure of \cite{Csiszar} with the objective function $D(\pi, q)$, where $\pi$ and $q$ ranges over $\mathcal{P}$ and $\mathcal{Q}$, respectively.
To see this, note that
$D(\pi_m, q_{m-1}) = D(f, f_{m-1}) \leq D(\pi, q_{m-1})$
for every $\pi \in \mathcal{P}$; the inequality holding since $f$ and $f_{m-1}$ are marginal densities of $\pi$ and $q_{m-1}$, respectively.
It follows that $\pi_m = \argmin_{\pi\in\mathcal{P}} D(\pi, q_{m-1})$.
Also, note that
\begin{align*}
D(\pi_m, q) &= \iint \pi_m(x,\theta) \log \frac{\pi_m(x,\theta)}{q(x,\theta)} \,\d\theta \,\d x \\
& = C - \int \log p(\theta) \int \pi_m(x,\theta) \,\d x \,\d\theta \\
&= C - \int \log p(\theta) \int \frac{f(x) q_{m-1}(x,\theta)}{\int q_{m-1}(x,\theta') d\theta'}\, \d x \,\d\theta \\
&= C - \int \log p(\theta) \int \frac{k(x,\theta) p_{m-1}(\theta)}{\int k(x,\theta') p_{m-1}(\theta')\,\d\theta'} f(x) \,\d x \,\d\theta \\
&= C - \int \log p(\theta) \,\d P_m(\theta),
\end{align*}
where $C$ does not depend on $p$.  The last integral is maximized at $p_m$, so $q_m = \argmin_{q\in\mathcal{Q}} D(\pi_m,q)$. Since $\mathcal{P}$ and $\mathcal{Q}$ are convex, we have
\[ D(f, f_{m-1}) = D(\pi_m, q_{m-1}) \searrow \inf_{p\in\mathcal{P}_0,\,q\in\mathcal{Q}} D(\pi, q), \]
by Theorem~3 in \cite{Csiszar}, with $\mathcal{P}_0 = \{\pi\in\mathcal{P}: D(\pi, q_m) < \infty \;\textrm{for some $m$} \}$.  We have $D(\pi,q) = D(f, f_p)$, and and the assumption $D(\pi_s^*, q_m) < \infty$ for some $m$ implies that
$\inf_{p\in\mathcal{P}_0}\inf_{q\in\mathcal{Q}} D(\pi, q) \leq \inf_P D(f, f_P)$.  
Since $D(f, f_m) \geq \inf_P D(f, f_P)$, we conclude that $D(f, f_m) \searrow \inf_P D(f, f_P)$.
\end{proof}

%\medskip

Two remarks are in order.  First, the integrability condition $D(\pi^*_s, q_m) < \infty$ for some $m \geq 0$ is not especially strong.  For example, assume there exists $P^*$ which minimizes $D(f, f_P)$ and has a density $p^*$.  Such a minimizer exists under a mild identifiability condition; see Lemma 3.1 of \cite{Kleijn}.  Then $D(\pi^*, q_m) < \infty$ for some $m$ implies the required integrability condition, where
$$
	\pi^*(x,\theta) = \frac{f(x) p^*(\theta) k(x,\theta)}{\int p^*(\theta') k(x,\theta') \d\theta'}.
$$
Second, if $\Theta$ is compact and the minimizer $P^*$ is unique, every subsequence of $(p_m)$ has a further subsequence weakly converging to $P^*$, so $(p_m)$ converges to $P^*$.

It is also worth noting that the monotonicity property of $D_m=D(f,f_m)$ from the theorem can be used to define a stopping rule.  For example, one could terminate algorithm \eqref{iter1} when $D_m$ itself and/or the difference $D_m-D_{m-1}$ falls below a certain user--specified tolerance.

\section{Smooth mixing density estimation} 

Suppose we observe $X_1,\ldots,X_n$ as i.i.d. from a distribution with density $f$ as in \eqref{fred}, and we wish to estimate a smooth version of $p$.  The nonparametric maximum likelihood estimator of $p$ is known to be discrete, so is not fully satisfactory for our purposes.  One idea for a smooth estimate of $p$ is to maximize a penalized likelihood.  The proposal in \cite{Liu} is to define a penalized log--likelihood function 
\[ \ell(\eta) = n^{-1} \sum_{i=1}^n \log \int k(X_i, \theta) e^{\eta(\theta)} \,\d\theta - \log \int e^{\eta(\theta)} \, \d\theta - \lambda\,\int [\eta''(\theta)]^2\,\d\theta, \]
%$$\ell(\eta)=n^{-1}\sum_{i=1}^n [\log k(x_i,\theta_i)+\eta(\theta_i)]-\log\int e^{\eta(\theta)}\,\d\theta-\lambda\,\int [\eta''(\theta)]^2\,\d\theta,$$
where $\lambda>0$ is a smoothing parameter, and $\eta$ determines the mixing density 
$$p(\theta)=\frac{e^{\eta(\theta)}}{\int e^{\eta(\theta)}\,\d\theta}.$$
The right-most integral in the expression for $\ell(\eta)$ measures the curvature of $\eta$, so maximizing $\ell$ will encourage solutions which are ``less curved,'' i.e., more smooth.  An EM algorithm is available to produce an estimate of $\eta$ and, hence, of $p$.  However, each iteration requires solving a non--trivial functional differential equation. 

Here we propose a more direct approach, namely, to use algorithm \eqref{iter1} with the true density $f$ replaced by, say, a kernel density estimate of $f$.  For some $h>0$, define the kernel estimate 
$$f^h(x)=\frac{1}{nh} \sum_{i=1}^n \phi\left(\frac{x-X_i}{h}\right) $$
where $\phi(x)=\exp(-\half x^2)/\sqrt{2\pi}$ is the standard normal density function. The bandwidth $h$ can be selected in a variety of ways.  One option is the traditional approach (\cite{Sheather}) of minimizing asymptotic mean integrated square error, which gives $h$ of order $n^{-1/5}$.  Another idea is to get the nonparametric maximum likelihood estimator $\widehat P$ of $P$, define the corresponding mixture 
$$\widehat{f}(x)=\int k(x,\theta)\,\d\widehat{P}(\theta)$$
and then choose the bandwidth $\widehat{h}=\arg \min d(f^h,\widehat{f})$, where $d$ is, say, the $L_1$ distance.  Finally, our proposal is to estimate $p$ by running algorithm \eqref{iter1} with $f$ replaced by $f^{\hat h}$.  Since $f^{\hat h}$ is smooth, so too will be our estimate of $p$.  

As a quick real--data illustration, consider the well--known galaxy data set, available in the MASS package in R (\cite{R}).  We estimate the mixture density $f$ with the Gaussian kernel method described above, using the default settings in {\tt density}.  Then we estimate the mixing density via the procedure just described above.  In this case, we used 25 iterations of \eqref{iter1} and the results are displayed in Figure~\ref{fig:galaxy}.  Panel~(a) shows $p_{25}$ from \eqref{iter1}, and Panel~(b) shows the data histogram, the kernel density estimator (solid), and the mixture corresponding to $p_{25}$.  Both densities in Panel~(b) fit the data well, and the fact that the two are virtually indistinguishable suggests that the density $p$ in Panel~(a) is indeed a solution to the mixture equation.   

\begin{figure}
\begin{center}
\subfigure[Mixing density]{\scalebox{0.5}{\includegraphics{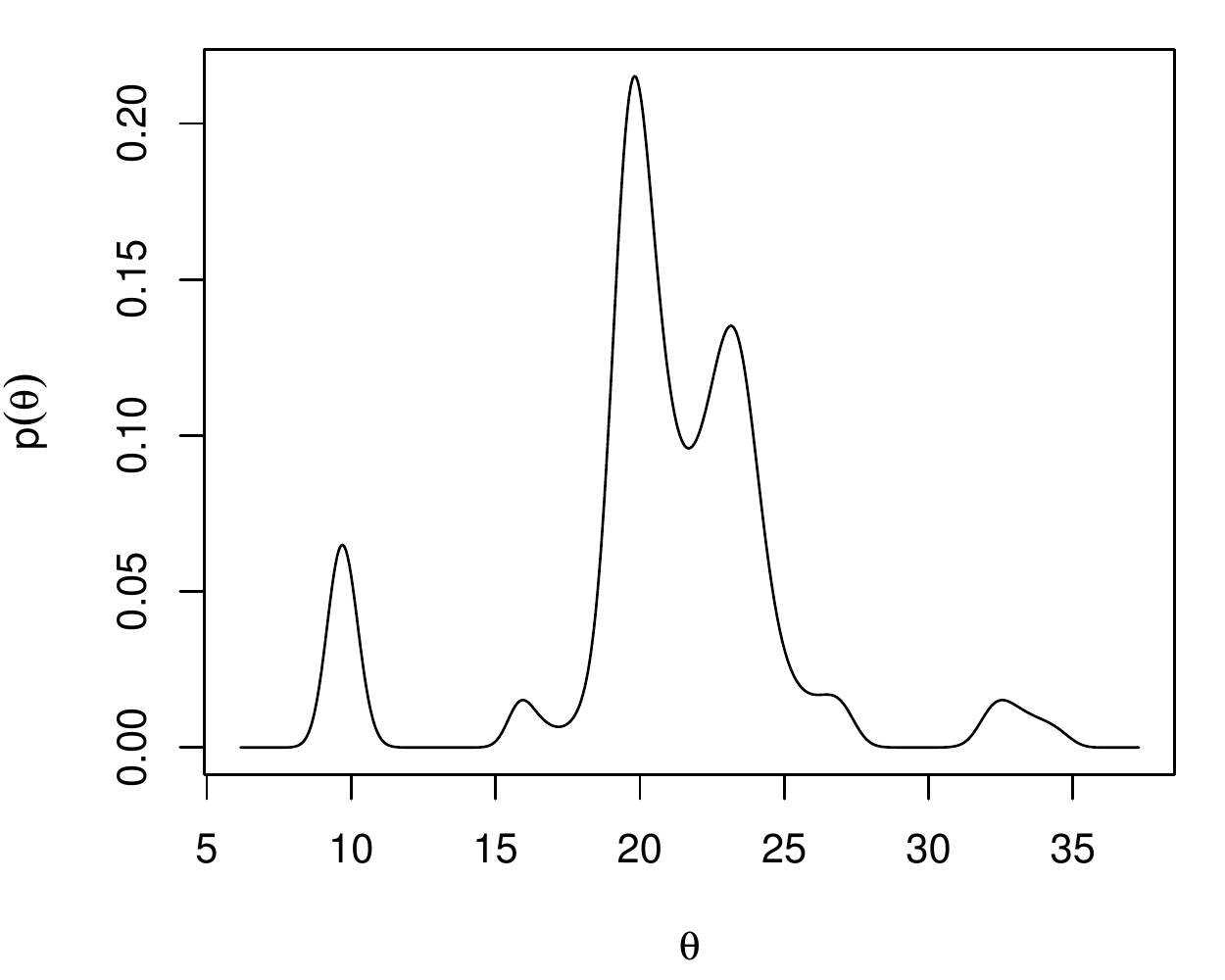}}}
\subfigure[Data and mixtures]{\scalebox{0.5}{\includegraphics{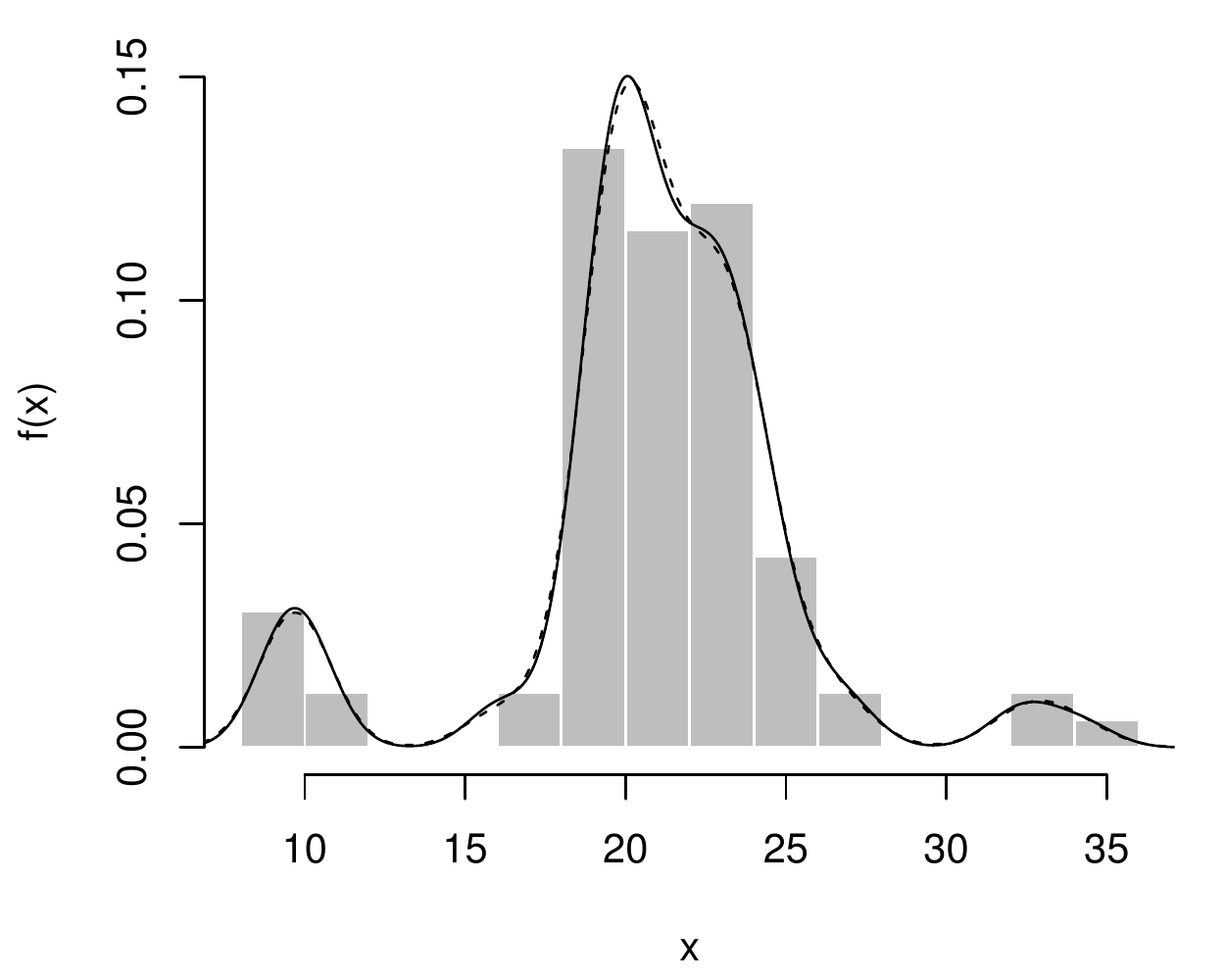}}}
\caption{Mixing and mixture density estimates based on algorithm \eqref{iter1} with a kernel density plug-in for the galaxy data.  In Panel~(b), solid line is the kernel estimate and dashed is the mixture corresponding to the mixing density in Panel~(b).}
\label{fig:galaxy}
\end{center}
\end{figure}

Next, for further investigation into the performance of the proposed mixing density estimation procedure, we present some examples with simulated data.  We start by considering a deconvolution problem, where $k(\cdot,\theta)$ is the normal density with mean $\theta$ and standard deviation $\sigma=0.05$.
Two mixing densities on the unit interval [0,1] are illustrated;
\begin{eqnarray*}
	p_1(\theta) &\propto& \theta^4 (1-\theta)^4
	\\
	p_2(\theta) &\propto& \phi \left( \frac{\theta-0.3}{0.1} \right) + 2 \phi \left( \frac{\theta-0.7}{0.1} \right).
\end{eqnarray*}
With a sample of size $n=300$, we first obtained a kernel density estimator $f^h$, using the Gaussian kernel, with bandwidth as in \cite{Sheather}, and then ran \eqref{iter1}.  
%We ran (\ref{iter1}), and the integral within is evaluated via the Gauss--Legendre quadrature rule with 300 quadrature points. 
As suggested above, monotonicity of $D_m=D(f^h,f_m)$ suggests a stopping rule, and we terminated the algorithm when $D_m-D_{m-1} < 10^{-5}$.  The number of iterations used were $m=8$ for $p_1$ and $m=14$ for $p_2$.  The estimates of the mixing and corresponding mixture densities are depicted in Figure~\ref{fig2}.  As can be seen in the right figures, the $f_m$ and $\widehat{f}$ are indistinguishable; that is, $D_m$ are effectively zero after a few iterations.

Next, we consider scale mixtures of normal distributions where $k(\cdot, \theta)$ is the centered normal density with variance $\theta$.  For the mixing density, two well--known distributions on the positive real line, the inverse-gamma and gamma densities, are considered, i.e.,
\[ p_1(\theta) \propto \theta^{-3} e^{-1/\theta} \quad \text{and} \quad p_2(\theta) \propto  e^{-5\theta}. \]
With a sample of size $n=300$, the estimator of the mixing density is obtained as previously.  The number of iterations are $m=57$ for $p_1$ and $m=78$ for $p_2$, and   the estimates of $p$ and $f$ are illustrated in Figure~\ref{fig3}.
Although the density estimates, $f^h$ and $f_m$, are close to the true $f$, there are deviations between $p$ and $p_m$, in particular for the inverse--gamma case. This is mainly due to the ill--posedness of the problem.

\begin{figure}[!htbp]
\begin{center}
\includegraphics[scale=1]{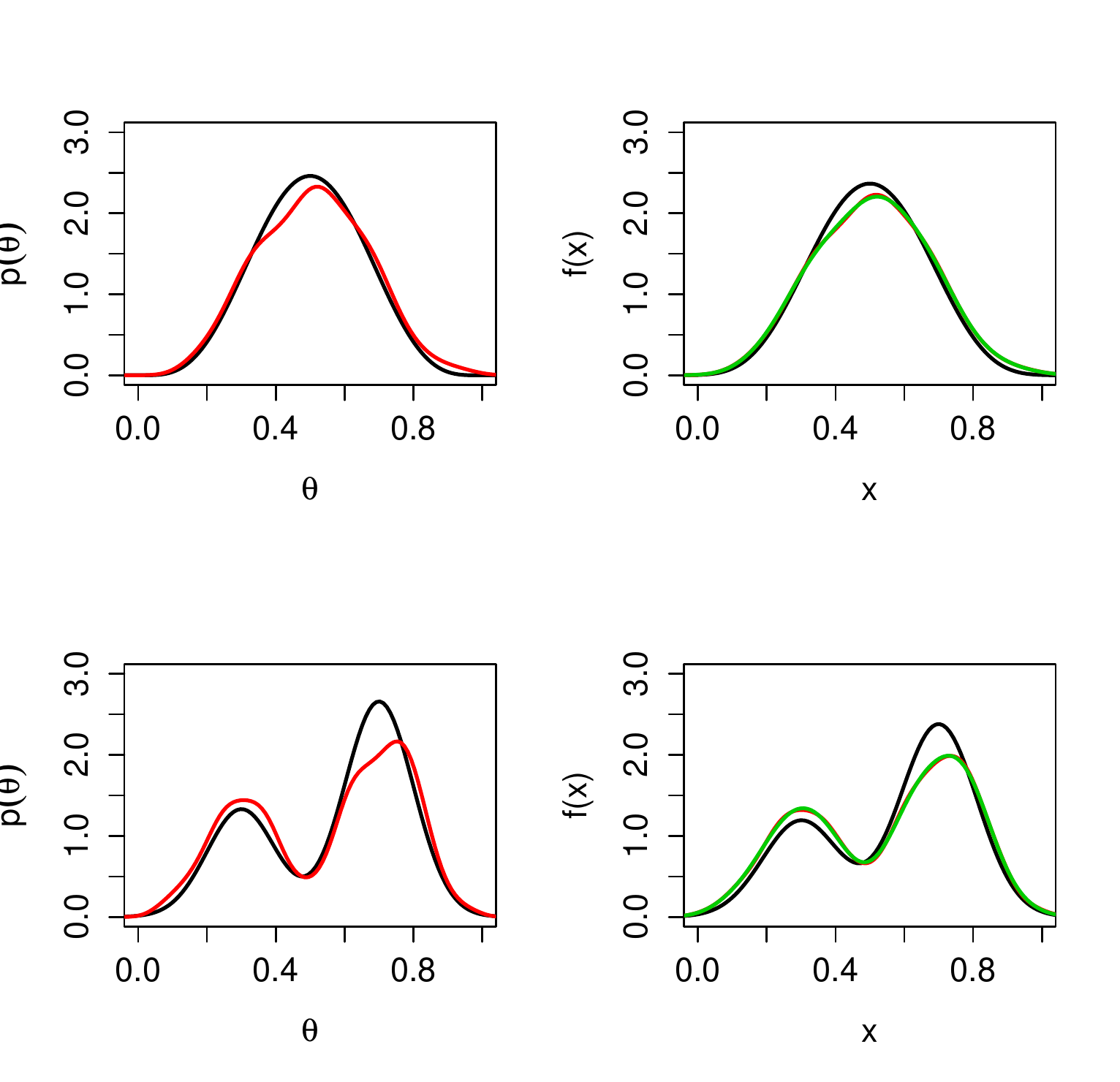}
\caption{Estimate of mixing density with location mixtures of normals. Left plots are the true (black) and estimated (red) mixing density. Right plots are the true (black), kernel estimator (red) and the one obtained by smooth NPMLE (green).}
\label{fig2}
\end{center}
\end{figure}

\begin{figure}[!htbp]
\begin{center}
\includegraphics[scale=1]{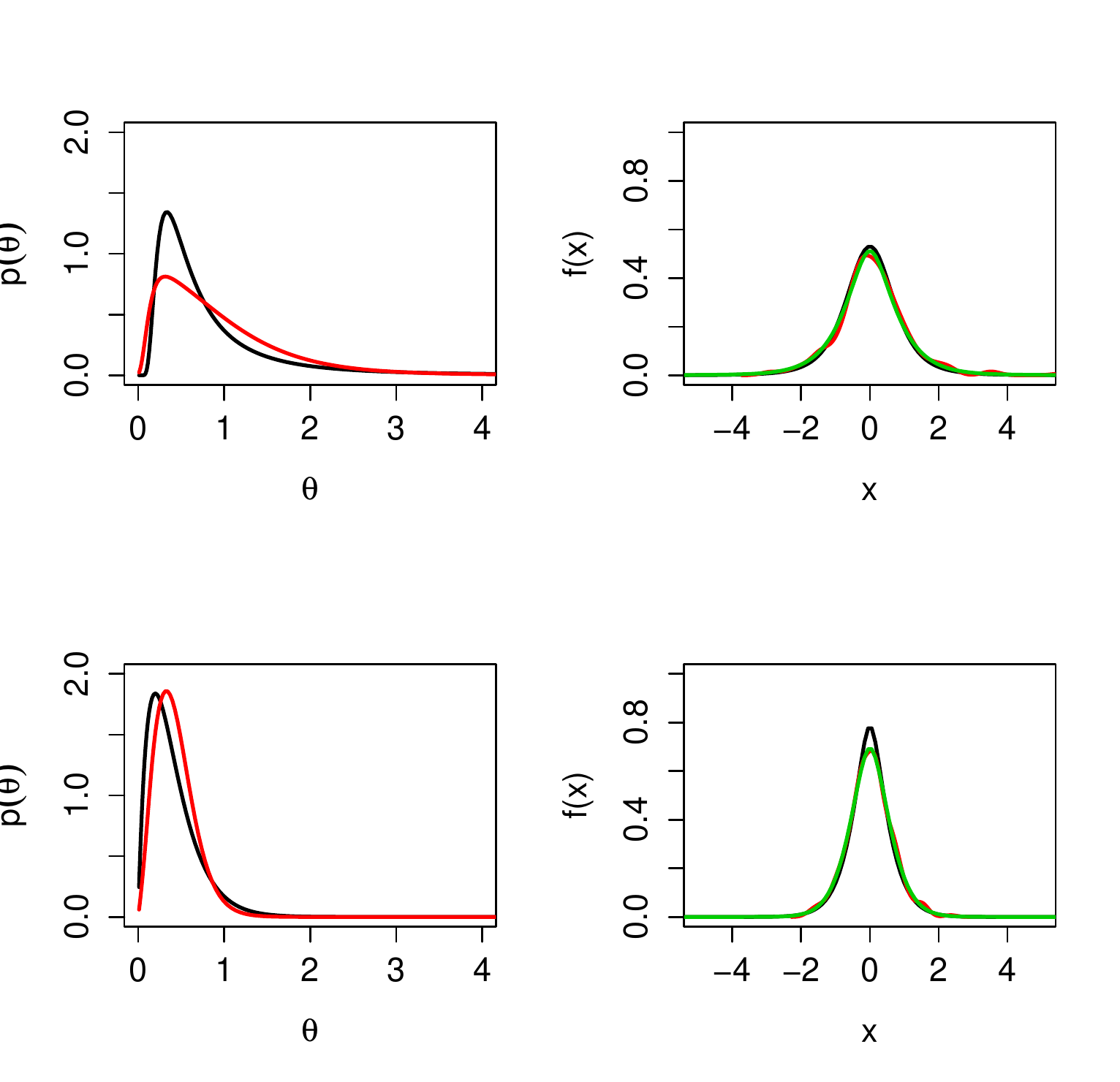}
\caption{Estimate of mixing density with scale mixtures of normals. Left plots are the true (black) and estimated (red) mixing density. Right plots are the true (black), kernel estimator (red) and the one obtained by smooth NPMLE (green).}
\label{fig3}
\end{center}
\end{figure}

\section{First passage time for Brownian motion} 

The second example is computing the stopping time density for Brownian motion hitting a boundary function.  For example, since \cite{Breiman} there has been substantial interest in the first passage time distribution of Brownian motion passing a square root boundary.  To set the scene, denote $(B(t))_{t\geq 0}$ as standard Brownian motion started at 0, let $h(t):(0,\infty)\to\RR$ be a continuous function with $h(0)>0$, the boundary function, and define
$$\tau=\inf\{t>0: B(t)\geq h(t)\}.$$
Then \cite{Peskir} provides an equation satisfied by the density $p$ of $\tau$, in particular, 
$$1-\Phi\left(\frac{h(t)}{\sqrt{t}}\right)=F(t)-\int_0^t \Phi\left(\frac{h(t)-h(s)}{\sqrt{t-s}}\right)\,\d P( s),$$
where $\Phi$ and $P$ are the probability measures/cumulative distribution functions corresponding to $\phi$ and $p$, respectively. This leads to a Volterra equation and \cite{Peskir} demonstrates solutions when $h$ is constant and linear. When $h$ is increasing and concave and $h(t)\leq h(0)+b\sqrt{t}$ then \cite{Peskir} proves that
$$P(t)=m(t)+\int_0^t R(t,s)\,m(s)\,\d s$$
where 
$$m(t)=2-2\Phi\left(\frac{h(t)}{\sqrt{t}}\right)\quad\mbox{and}\quad R(t,s)=\sum_{m=1}^\infty K_m(t,s), $$
with 
$$K_m(t,s)=\int_{s}^t K_1(t,r)\,K_{m-1}(r,s)\,\d r, $$
and
$$K_1(t,s)=\frac{1}{\sqrt{t-s}}\,\phi\left(\frac{h(t)-h(s)}{\sqrt{t-s}}\right)\,\left(2h'(s)-\frac{h(t)-h(s)}{t-s}\right).$$
Clearly $h(t)=a+b\sqrt{t}$ meets the requirements for this result. 

On the other hand, we can solve using (\ref{iter1}) after setting up a suitable Fredholm equation for $h(t)$ increasing and bounded by $a+b\sqrt{t}$, but not necessarily concave. Define, for any $x>0$, the martingale
$$X_t=\exp\left\{x\,B(t)-\half\,x^2 t\right\}.$$
With $\theta=T_{a,b}=\inf\{t>0: B(t)\geq a+b\,h(t)\}$ with $a,b>0$, $h(0)=0$ and $h(t)\leq \sqrt{t}$, we have that $X_{t\wedge T_{a,b}}$ is a martingale and, from the optional stopping theorem, that
$E(X_{t\wedge T_{a,b}})=E(X_0)=1$.

Now it is easy to show that $X_{t\wedge T_{a,b}}\leq \exp\{x a+\half b^2h^2(t)/t\}$, which is bounded and, hence, from the dominated convergence theorem, 
\[ E(X_{T_{a,b}})=\lim_{t\to\infty} E(X_{t\wedge T_{a,b}})=E(X_0)=1. \]
Effectively, we are also using the result that $T_{a,b}<\infty$ almost surely. Hence,
$$\int \exp\left\{x a+x bh(\theta)-\half\,x^2\theta\right\}\,p(\theta)\,\d\theta=1$$
which after some algebra yields the Fredholm equation
\begin{equation}\label{sqrt}
\int_0^\infty k(x, \theta)\,\widetilde{p}(\theta)\,\d\theta=a\,e^{-x a}=f(x),
\end{equation}
where $k(x, \theta)$ is the normal density with mean $b\,h(\theta)/\theta$ and variance $1/\theta$, constrained on $\RR_+$, and
$$a\,\sqrt{2\pi}\,e^{\half b^2h^2(\theta)/\theta}\,\Psi(-bh(\theta)/\sqrt{\theta})\,p(\theta)=\sqrt{\theta}\,\widetilde{p}(\theta)$$
and $\Psi(\cdot)=1-\Phi(\cdot)$. Now \eqref{sqrt} can be readily solved by algorithm \eqref{iter1}.

The approach here is closely related to work appearing in \cite{Valov}. This considers first passage times which are almost surely finite and of the type $\tau=\inf\{t>0: B(t)\leq b_\alpha(t)=b(t)+\alpha\,t\}$, where $b_\alpha(t)>-c$ for some $c>0$. While \cite{Valov} constructs a Fredholm equation using Girsanov's theorem, as we do, this is quite different and each equation depends on the type of boundary, whereas ours is always based on the half--normal and exponential distributions.  

Here we demonstrate the algorithm for solving the Fredholm equation for the hitting time density with a square root boundary; specifically
$a+b\sqrt{t}$ with $a=1$ and $b=0.1$. We start with a grid of points $\theta_j=hj$, where $h=0.05$ and $j=1,\ldots,1000$. The starting density is $p_0(\theta)=0.01\,e^{-0.01\theta}$ and at each iteration we evaluate $p_m(\theta_j)$. Due to the ease of sampling from an exponential density, we evaluate
$$\int_0^\infty \frac{k(x, \theta_j)\,f(x)}{f_{m}(x)}\,\d x$$
using Monte Carlo methods, i.e., 
$$N^{-1}\sum_{i=1}^N \frac{k(X_i, \theta_j)}{f_{m}(X_i)}$$
where the $(X_i)$ are iid from the exponential distribution with mean $1/a$, and we take $N=5000$.  The $f_m(X_i)$ is evaluated using the trapezoidal rule using the $\widetilde{p}_m(\theta_j)$ values, so
$$f_m(X_i)\approx \frac{h}{2} \sum_{j=1}^M \left[k(X_i, \theta_{j-1})\,\widetilde{p}_m(\theta_{j-1})+k(X_i, \theta_{j})\,\widetilde{p}_m(\theta_{j})\right],$$
where $\theta_0=0$. Consequently, we have numerically, 
$$\widetilde{p}_{m+1}(\theta_j)=N^{-1}\sum_{i=1}^N \frac{2\,k(X_i, \theta_j)\,\widetilde{p}_m(\theta_j)}{h \sum_{j=1}^M \left[k(X_i, \theta_{j-1})\,\widetilde{p}_m(\theta_{j-1})+k(X_i, \theta_{j})\,\widetilde{p}_m(\theta_{j})\right]}.$$
The estimated $p(\theta)$, obtained by transforming $\widetilde{p}(\theta)$, is given in Figure~\ref{fig1}, based on 200 iterations of algorithm \eqref{iter1}.

\begin{center}
\begin{figure}[t]
\begin{center}
\includegraphics[scale=0.45]{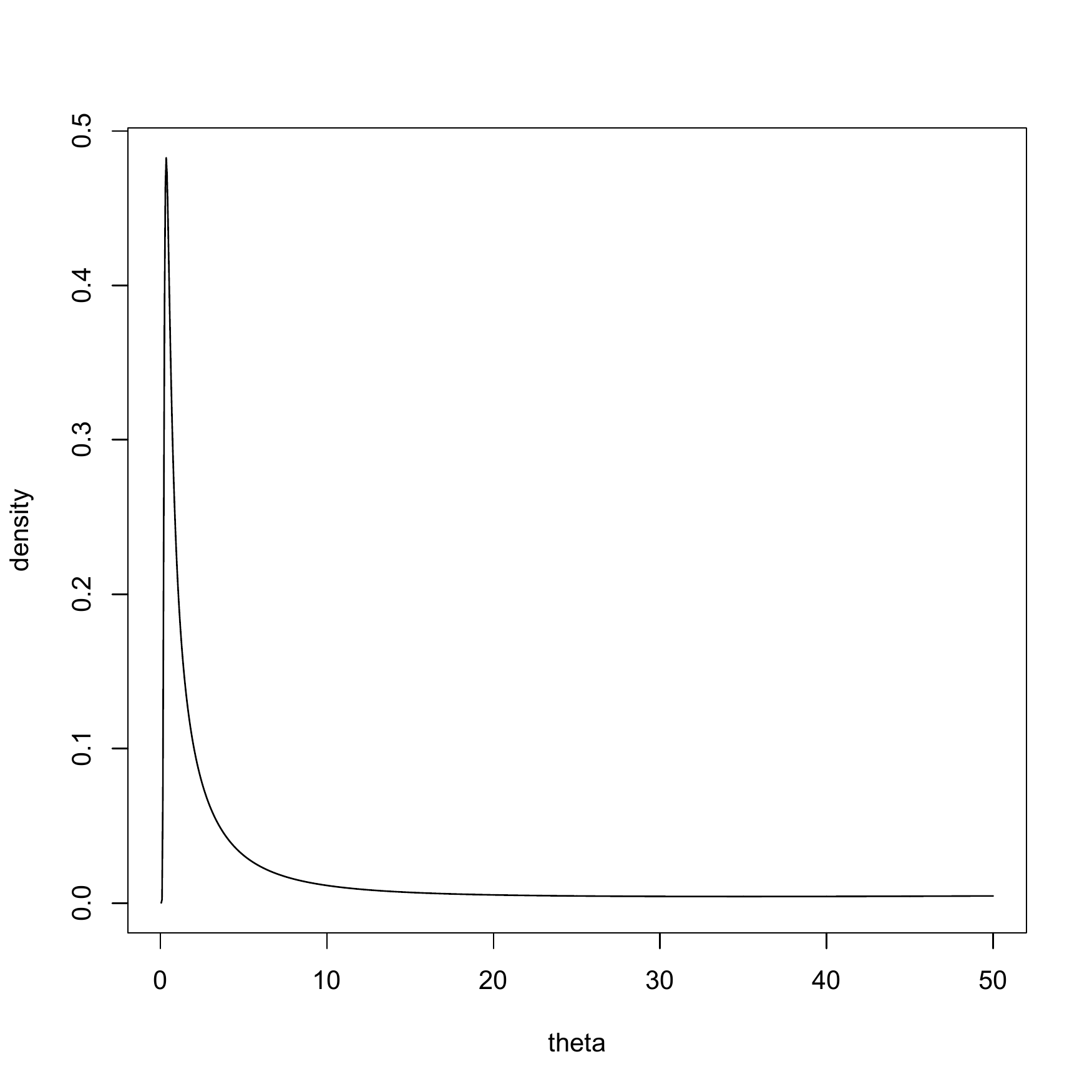}
\caption{Estimate of hitting time density for square root boundary}
\label{fig1}
\end{center}
\end{figure}
\end{center}

\section{General Fredholm equation}

Algorithm \eqref{iter1} can be applied to solve general Fredholm equations of the first kind, where $f, p$ and $k(\cdot,\theta)$ are not necessarily probability density functions.  Assume first that $f, p$ and $k(\cdot,\theta)$ in \eqref{fred} are non-negative functions, but not necessarily probability densities.  Then equation \eqref{fred} can be rewritten as
$$
f(x)=\int_\Theta \widetilde k(x,\theta)\, q(\theta)\,\d\theta,
$$
where $\widetilde k(x,\theta) = k(x,\theta) / \int k(x,\theta) \d x$ and $q(\theta) = p(\theta) \int k(x,\theta) \d x$.
Thus, the update (\ref{iter1}) can be replaced with
\begin{equation} \label{iter3}
p_m(\theta)=\frac{p_{m-1}(\theta)}{\int_{\mathbb{X}} k(x,\theta)\, \d x}\, \int_{\mathbb{X}}\frac{k(x,\theta)\,f(x)}{f_{m-1}(x)}\,\d x,
\end{equation}
which was also considered in \cite{Vardi}. Here it is assumed that $\sup_x \int k(x,\theta)\d\theta$ and $\sup_\theta\int k(x,\theta)\,\d x$ are both finite.

Next, assume that the solution $p$ and $f$ may not be necessarily non-negative functions.  In this case, instead of the original equation (\ref{fred}), we solve an equivalent non-negative equation
\begin{equation} \label{fred2}
\widetilde f(x)=\int_\Theta k(x,\theta)\,\widetilde p(\theta)\,\d\theta,
\end{equation}
where $\widetilde p(\theta) = p(\theta) + t$, $\widetilde f(x) = f(x) + t\int_\Theta k(x,\theta) \,\d\theta$, and $t > 0$ is a constant to be specified.
Since $k$ is non-negative, so are both $\widetilde f$ and $\widetilde p$ for a sufficiently large $t$.  As illustrated below, the value of $t$ rarely affects the convergence rate in practice.  Therefore, $t$ can be chosen as a very large number.

Finally, assume that $k$ is not necessarily non-negative, and write $k = k^+ - k^-$, where $k^+, k^- \geq 0$.  This case can also be solved by transforming the original equation to a non-negative one.  For convenience, assume that $\Theta = [0,1]$, then the original equation (\ref{fred}) can be written as
$$
	f(x)=\int_0^1 k^+(x,\theta) p(\theta) \d\theta - \int_0^1 k^-(x,\theta) p(\theta) \d\theta.
$$
Note that
\begin{equation} \label{aux-eq}
	\int_0^1 k^-(x,\theta) p(\theta) \d\theta - \int_1^2 k^-(x,\theta-1) p(\theta-1) \d\theta = 0,
\end{equation}
and therefore, by adding the last two display equations, we have
$$
   f(x)=\int_0^1 k^+(x,\theta) p(\theta) \d\theta  + \int_1^2 k^-(x,\theta-1) \{-p(\theta-1)\}\, \d\theta.
$$
Let
\begin{eqnarray*}
	\widetilde k(x,\theta) = \left\{ \begin{array}{ll}
		k^+(x,\theta) & \textrm{if $\theta \in [0,1]$}
		\\
		k^-(x,\theta-1) & \textrm{if $\theta \in (1,2]$}
	\end{array} \right.
\end{eqnarray*}
and 
\begin{eqnarray*}
	\widetilde p(\theta) = \left\{ \begin{array}{ll}
		p(\theta) & \textrm{if $\theta \in [0,1]$}
		\\
		-p(\theta-1) & \textrm{if $\theta \in (1,2]$},
	\end{array} \right.
\end{eqnarray*}
then we have the  Fredholm equation 
\begin{equation} \label{fred3}
	f(x)=\int_0^2 \widetilde k(x,\theta) \widetilde p(\theta) \d\theta
\end{equation}
with a non-negative kernel, which can be solved as previously described.

Note that $\widetilde k$ and $\widetilde p$ may have discontinuities; but this does not cause any problem once (\ref{fred3}) has the unique solution.  If there is another solution to (\ref{fred3}), say $\overline p$, and (\ref{aux-eq}) is not satisfied, the restriction of $\overline p$ on [0,1] may not be a solution of the original equation.  In this case we may apply another decomposition of $k$. In many examples, however, the simple approach (\ref{fred3}) works well.

For illustration, consider the equation (\ref{fred}), where $k(\cdot,\theta)$ is the normal density with mean $\theta$ and standard deviation $\sigma=0.05$, and $p$ has both positive and negative components.  In particular, we consider two examples
\begin{eqnarray*}
	p_1(\theta) &=& b_{2,5}(\theta) - b_{4,1}(\theta)
	\\
	p_2(\theta) &=& b_{10,1}(\theta) - b_{1,10}(\theta),
\end{eqnarray*}
where $b_{a,b}(\cdot)$ is the density of Beta$(a,b)$ distribution.  In both cases, a transformed equation (\ref{fred2}) is solved with $t=50$ and the results are illustrated in Figure~\ref{fig4}.  It can be easily seen that the true solution and $p_m$ are nearly the same.  Different values of $t$ ($5 \times 10^k$ with $k\leq 4$) have been tried, and in any case, the algorithm has been stopped in 10 iterations yielding nearly the same solution.

\begin{figure}[!htbp]
\begin{center}
\includegraphics[scale=1]{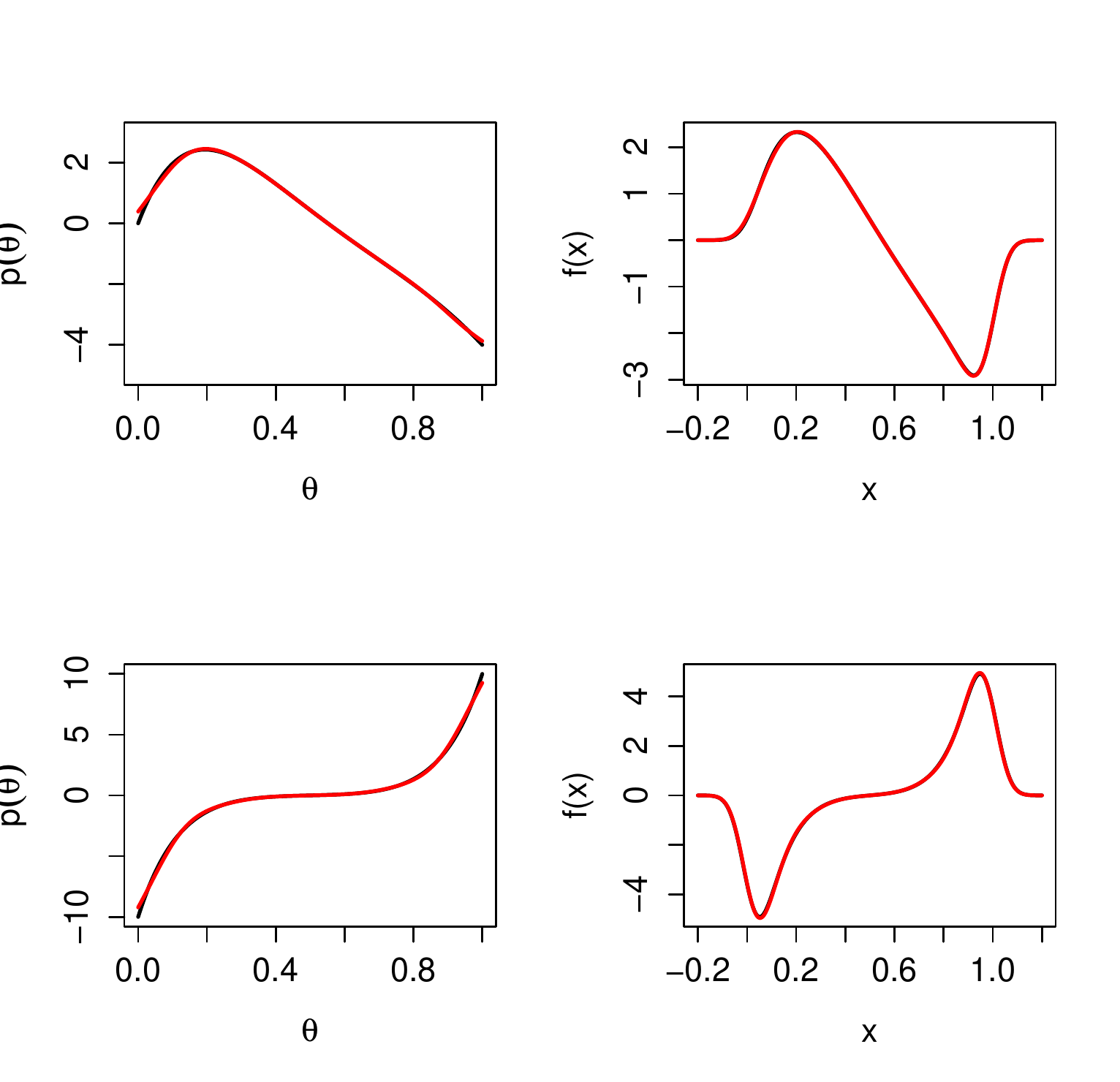}
\caption{Illustration with a positive kernel $k(x,\theta) = \phi_\sigma(x-\theta)$ and general solutions $p_1(\cdot)$ (top) and $p_2(\cdot)$ (bottom).}
\label{fig4}
\end{center}
\end{figure}

\begin{figure}[!htbp]
\begin{center}
\includegraphics[scale=1]{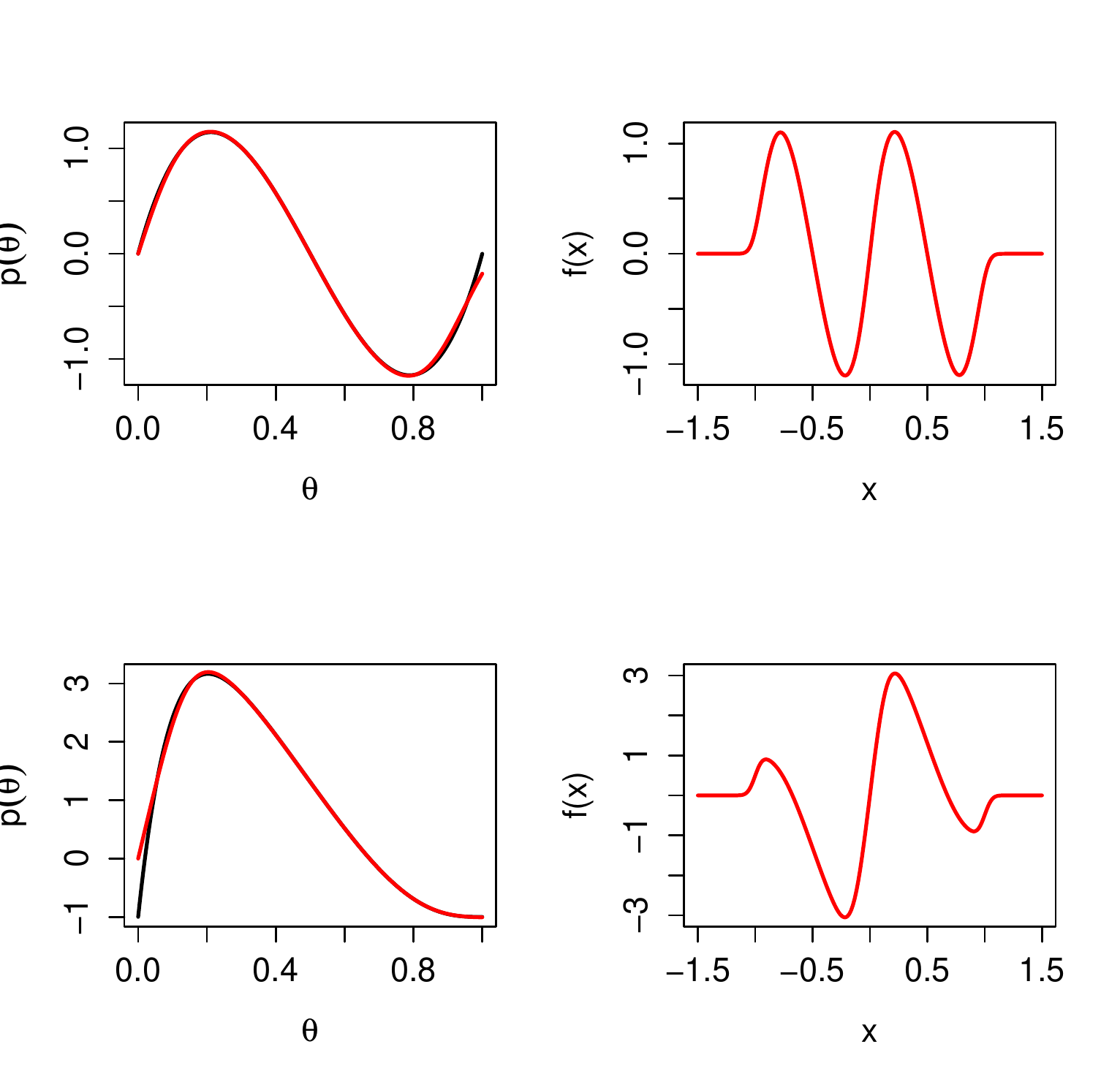}
\caption{Illustration with a general kernel $k(x,\theta) = \phi_\sigma(x-\theta) - \phi_\sigma(x+\theta)$ for $p_1(\cdot)$ (top) and $p_2(\cdot)$ (bottom).}
\label{fig5}
\end{center}
\end{figure}

Next, we consider a general kernel $k(x,\theta) = \phi_\sigma(x-\theta) - \phi_\sigma(x+\theta)$ with $\sigma=0.05$ and
\begin{eqnarray*}
	p_1(\theta) &=& b_{2,3}(\theta) - b_{3,2}(\theta)
	\\
	p_2(\theta) &=& b_{2,7}(\theta) + b_{3,4}(\theta) - 1,
\end{eqnarray*}
where $\phi_\sigma(x) = \phi(x/\sigma)/\sigma$.  The equations are solved by setting $k^+(x,\theta) = \phi_\sigma(x-\theta)$, $k^-(x,\theta) = \phi_\sigma(x+\theta)$ and $t=50$.  For both examples, the algorithm stopped in 5 iterations.  Results are illustrated in Figure~\ref{fig5}.

\section{Conclusion}

In this paper, we focused on an algorithm for solving Fredholm integral equations of the first kind, its properties, and some applications.  For the mixing density estimation application described in Section~4, we did not address the question of whether the estimate based on plugging in a kernel density estimator for $f$ in \eqref{iter1} would be consistent in the statistical estimation sense.  The {\em predictive recursion} method of \cite{Newton} can also quickly produce a smooth estimator of the mixing density, and it was shown in \cite{Tokdar} and \cite{Martin} that the estimator is consistent, but non-standard arguments are needed because of its dependence on the data order.  We are optimistic that the estimator described in Section~4, through the simple formula \eqref{iter1} for the updates, and the well-known behavior of the kernel density estimator, can have even stronger convergence properties than those demonstrated for predictive recursion.  

In Section 5 we were able to extend the class of boundary function for which the hitting time density can be solved using a novel Fredholm equation. Future work would consider upper and lower boundaries. Section 6 we were able to extend the basic algorithm, which was set up for density functions, to non--negative functions.

\bigskip
\begin{center}
{\large\bf SUPPLEMENTAL MATERIALS}
\end{center}

\begin{description}

\item[Rcpp code:] It contains code to perform the iterative algorithm proposed in this paper and numerical experiments.

\end{description}

\bibliographystyle{apalike}
\bibliography{bibliography}

\end{document}